\documentclass[12pt, reqno]{amsart}
\usepackage{ amsmath,amsthm, amscd, amsfonts, amssymb, graphicx, color}
\usepackage[bookmarksnumbered, colorlinks, plainpages]{hyperref}
\newtheorem{thm}{Theorem}[section]
\newtheorem{cor}[thm]{Corollary}
\newtheorem{lem}[thm]{Lemma}
\newtheorem{prop}[thm]{Proposition}

\newtheorem{exam}[thm]{Example}
\numberwithin{equation}{section}

\begin{document}

\title{Generalized Hirano inverses in rings}

\author{Marjan Sheibani Abdolyousefi}
\author{Huanyin Chen}
\address{
Women's University of Semnan (Farzanegan), Semnan, Iran}
\email{<sheibani@fgusem.ac.ir>}
\address{
Department of Mathematics\\ Hangzhou Normal University\\ Hang -zhou, China}
\email{<huanyinchen@aliyun.com>}


\begin{abstract}
We introduce and study a new class of generalized inverses in rings.
An element $a$ in a ring $R$ has generalized Hirano inverse if there exists $b\in R$ such that $bab=b, b\in comm^2(a), a^2-ab\in R^{qnil}.$ We prove that the generalized Hirano inverse of an element is its generalized Drazin inverse. An element $a\in R$ has generalized Hirano inverse if and only if there exists $p=p^2\in comm^2(a)$ such that $a^2-p\in R^{qnil}$. We then completely determine when a $2\times 2$ matrix over projective-free rings has generalized Hirano inverse. Cline's formula and additive properties for generalized Hirano inverses are thereby obtained.
\end{abstract}

\maketitle

\section{Introduction}

Let $R$ be an associative ring with an identity. The commutant of $a\in R$ is defined by $comm(a)=\{x\in
R~|~xa=ax\}$. Set $R^{qnil}=\{a\in R~|~1+ax\in U(R)~\mbox{for
every}~x\in comm(a)\}$. We say $a\in R$ is quasinilpotent if
$a\in R^{qnil}$. For a Banach algebra $A$ it is well known
 that (see~\cite[page 251]{H}) $$a\in A^{qnil}\Leftrightarrow
\lim\limits_{n\to\infty}\parallel a^n\parallel^{\frac{1}{n}}=0.$$ The double commutant of $a\in R$ is defined
by $comm^2(a)=\{x\in R~|~xy=yx~\mbox{for all}~y\in comm(a)\}$.  An element $a$ in $R$ is said to have generalized Drazin inverse if there exists $b\in R$ such that $$b=bab, b\in comm^2(a), a-a^2b\in R^{qnil}.$$  The preceding $b$ is unique, if such element exists. As usual,
it will be denoted by $a^{d}$, and called the generalized Drazin inverse of $a$. Generalized Drazin inverse is extensively studied in both matrix theory and Banach algebra (see~\cite{C, LCC, KP, LCC2} and ~\cite{LZ}).

The goal of this paper is to introduce and study a new class of generalized inverses in a ring.
An element $a\in R$ has generalized Hirano inverse if there exists $b\in R$ such that $$b=bab, b\in comm^2(a), a^2-ab\in R^{qnil}.$$  We shall prove that the preceding $b$ is unique, if such element exists. It will be denoted by $a^{h}$, and called the generalized Hirano inverse of $a$.

In Section 2 we prove that the generalized Hirano inverse of an element is its generalized Drazin inverse. An element $a\in R$ has generalized Hirano inverse if and only if there exists $p=p^2\in comm^2(a)$ such that $a^2-p\in R^{qnil}$.

Recall that a commutative ring $R$ is projective-free
provided that every finitely generated projective $R$-module is
free. The class of projective-free rings is very large. It includes local rings, principle ideal domains, polynomial rings over every principal ideal domain, etc. In Section 3, we completely determine when a $2\times 2$ matrix over projective-free rings has Hirano inverse. Let $R$ be
a projective-free ring. We prove that $A\in M_2(R)$ has generalized Hirano inverse if and only if $det(A), tr(A^2)\in J(R)$; or $det^2(A)\in 1+J(R), tr(A^2)\in 2+J(R)$; or $det(A)\in J(R), tr(A^2)\in 1+J(R)$ and $x^2-tr(A^2)x+det^2(A)=0$ is solvable.

Let $a,b\in R$. Then $ab$ has generalized Drazin inverse if and only if $ba$ has generalized Drazin inverse, and
$(ba)^d = b((ab)^d)^2a$. This was known as Cline's formula for generalized Drazin inverses.
It plays an important role in
revealing the relationship between the generalized Drazin inverse of a sum of two elements and the
generalized Drazin inverse of a block triangular matrix of the form. In Section 4, we establish Cline's formula for generalized Hirano inverses. Let $R$ be a ring, and let $a,b,c\in R$. If $aba=aca$, we prove that $ac$ has generalized Hirano inverse if and only if $ba$ has generalized Hirano inverse, and that $(ba)^{h}=b((ac)^{h})^2a$.

Finally, we consider additive properties of such generalized inverses in a Banach algebra. Let $A$ be a Banach algebra and $a,b\in A$. If $a,b\in A$ have generalized Hirano inverses and satisfy
$$a=ab^{\pi}, b^{\pi}ba^{\pi}=b^{\pi}b, b^{\pi}a^{\pi}ba=b^{\pi}a^{\pi}ab,$$ we prove that $a+b$ has generalized Hirano inverse. Here $a^{\pi}=1-aa^h$ and $b^{\pi}=1-bb^{h}$.

Throughout the paper, all rings are associative with an identity and all Banach algebras are complex. We use $N(R)$ to denote the set of all nilpotent elements in $R$ and $J(R)$ to denote the Jacobson radical of $R$. $\chi (A)=det(tI_n-A)$ is the characteristic polynomial of $n\times n$ matrix $A$. ${\Bbb N}$ stands for the set of all natural numbers.

\section{Generalized Drazin inverses}

The goal of this section is to explore the relations between generalized Drazin inverses and generalized Hirano inverses.
We begin with

\begin{lem} Let $R$ be a ring and $a\in R^{qnil}$. If $e^2=e\in comm(a)$, then $ae\in R^{qnil}$.\end{lem}
\begin{proof} Let $x\in comm(ae)$. Then $xae=aex$, and so $(exe)a=ex(ae)=eaex=aex=a(exe)$, i.e., $exe\in comm(a)$. Hence, $1-a(exe)\in U(R)$, and so
$1-(ae)x\in U(R)$. This completes the proof.\end{proof}

\begin{thm} Let $R$ be a ring and $a\in R$. If $a$ has generalized Hirano inverse, then it has generalized Drazin inverse.\end{thm}
\begin{proof} Suppose that $a$ has generalized Hirano inverse. Then we have $x\in R$ such that $a^2-ax\in R^{qnil}, x\in comm^2(a)$ and $xax=x$. As $a\in comm(a)$, we see that
$ax=xa$. Hence, $a^2-a^2x^2=a^2-a(xax)=a^2-ax\in R^{qnil}$, and so $$a^2(1-a^2x^2)=(a^2-a^2x^2)(1-a^2x^2)\in R^{qnil},$$ by Lemma 2.1.  As $1-a^2x^2=1-ax\in R$ is an idempotent and $1-ax\in comm(a^2-a^2x^2)$, it follows by Lemma 2.1 that $$a(a-a^2x)=a(a-a^2(ax^2))=a^2(1-a^2x^2)\in R^{qnil}.$$ By using Lemma 2.1 again, $(a-a^2x)^2=a(a-a^2x)(1-ax)\in R^{qnil}$.
Let $y\in comm(a-a^2x)$. Then $y^2\in comm (a-a^2x)^2$. Thus, $1-(a-a^2x)^2y^2\in U(R)$. We infer that $(1-(a-a^2x)y)(1+(a-a^2x)y)\in U(R)$. Therefore $1-(a-a^2x)y\in U(R)$, and so $a-a^2x\in R^{qnil}$. As $ax=xa$ and $xax=x$, we conclude that $a$ has generalized Drazin inverse $x$.\end{proof}

\begin{cor} Let $R$ be a ring and $a\in R$. Then $a$ has at most one generalized Hirano inverse in $R$, and if the generalized Hirano inverse
of $a$ exists, it is exactly the generalized Drazin inverse of $a$, i.e., $a^{h}=a^{d}$.\end{cor}
\begin{proof} In view of ~\cite[Theorem 4.2]{KP}, every element has at most one generalized Drazin inverse, we obtain the result by Theorem 2.1.\end{proof}

\begin{lem} Let $R$ be a ring and $a\in R$. Then the following are equivalent:\end{lem}
\begin{enumerate}
\item [(1)]{\it $a$ has generalized Hirano inverse.}
\vspace{-.5mm}
\item [(2)]{\it There exists $b\in R$ such
that
$$b=ba^2b, b\in comm^2(a), a^2-a^2b\in R^{qnil}.$$}
\end{enumerate}
\begin{proof} $\Longrightarrow $ By hypothesis, there exists $b\in R$ such that $$bab=b, b\in comm^2(a), a^2-ab\in R^{qnil}.$$ As $a\in comm(a)$, we see that
$ab=ba$. Hence, $$a^2-a^2b^2=a^2-a(bab)=a^2-ab\in R^{qnil}, b^2a^2b^2=(bab)^2=b^2.$$
Let $xa=ax$. Then $x\in comm(a)$; hence, $bx=xb$. Thus, $b^2x=xb^2$, and so
$b^2\in comm^2(a)$, as required.

$\Longleftarrow$ By hypothesis, there exists $b\in R$ such that $$ba^2b=b, b\in comm^2(a), a^2-a^2b\in R^{qnil}.$$ Thus,
$$a^2-a(ab)=a^2-a^2b\in R^{qnil}, (ab)a(ab)=a(ba^2b)=ab.$$
Let $xa=ax$. Then $xb=bx$, and so $(ab)x=a(bx)=a(xb)=(ax)b=x(ab)$.
Hence, $ab\in comm^2(a)$. Accordingly, $a\in R$ has generalized Hirano inverse.\end{proof}

\begin{thm} Let $R$ be a ring and $a\in R$. Then the following are equivalent:\end{thm}
\begin{enumerate}
\item [(1)]{\it $a$ has generalized Hirano inverse.}
\vspace{-.5mm}
\item [(2)]{\it There exists $p^2=p\in comm^2(a)$ such that $a^2-p\in R^{qnil}$.}
\end{enumerate}
\begin{proof} $(1)\Rightarrow (2)$ By Lemma 2.4, there exists $b\in R$ such
that $$b=ba^2b, b\in comm^2(a), a^2-a^2b\in R^{qnil}.$$ Set $p=a^2b$. Then $p^2=p\in R$.
Let $xa=ax$. Then $xb=bx$, and so $xp=x(a^2b)=(a^2b)x$. hence, $p\in comm^2(a)$. Furthermore, $a^2-p\in R^{qnil}$, as required.

$(2)\Rightarrow (1)$ Since $a^2-p\in R^{qnil}$, we see that $a^2+1-p\in U(R)$. Set $e=1-p$ and $b=(a^2+1-p)^{-1}(1-e)$.
Let $xa=ax$. Then $px=xp$ and $xa^2=a^2x$, and so $x(a^2+1-p)=(a^2+1-p)x$. This implies that $x(ab)=(ab)x$, i.e., $ab\in comm^2(a)$.
We check that $$\begin{array}{lll}
a^2-a(ab)&=&a^2-a^2(a^2+1-p)^{-1}(1-e)\\
&=&a^2-(a^2+e)(a^2+1-p)^{-1}(1-e)\\
&=&a^2-p\\
&\in& R^{qnil}.
\end{array}$$
Clearly, $a^2+e\in U(R)$, and so $$\begin{array}{lll}
(ab)a(ab)&=&a^3(a^2+1-p)^{-2}(1-e)\\
&=&a(a^2+e)(a^2+1-p)^{-2}(1-e)\\
&=&a(a^2+e)^{-1}(1-e)\\
&=&ab.
\end{array}$$
Therefore $a$ has generalized Hirano inverse $ab$, as asserted.\end{proof}

\begin{cor} Let $R$ be a ring and $a\in R$. Then the following are equivalent:\end{cor}
\begin{enumerate}
\item [(1)]{\it $a$ has generalized Hirano inverse.}
\vspace{-.5mm}
\item [(2)]{\it There exists $b\in R$ such
that
$$ab=(ab)^2, b\in comm^2(a), a^2-ab\in R^{qnil}.$$}
\end{enumerate}
\begin{proof} $(1)\Rightarrow (2)$ This is obvious.

$(2)\Rightarrow (1)$ Set $p=ab$. If $ax=xa$, then $bx=xb$; hence, $px=(ab)x=axb=x(ab)=xp$. Thus, $p\in comm^2(a)$.
Since $a^2-p\in R^{qnil}$, we complete the proof by Theorem 2.5.\end{proof}

In light of Theorem 2.5 and~\cite[Theorem 3.3]{CS}, we easily prove that every element in a ring $R$ is the sum of a tripotent $e$ (i.e., $e^3=e$) and a nilpotent that commute if and only if every element in $R$ has generalized Drazin inverse and $R^{qnil}=N(R)$. This also gives the link between generalized Drazin inverses and tripotents in a ring (see~\cite{HT})

\begin{prop} Let $R$ be a ring. If $a\in R$ has generalized Hirano inverse, then there exists a unique idempotent $p\in R$ such
that
$$pa=ap, a^2-p\in R^{qnil}.$$\end{prop}
\begin{proof}  In view of Theorem 2.5, there exists an idempotent $p\in R$ such
that $$p\in comm^2(a), a^2-p\in R^{qnil}.$$
Suppose that there exists an idempotent $q\in R$ such
that
$$qa=aq, a^2-q\in R^{qnil}.$$ Let $e=1-p$ and $f=1-q$. Then $a^2-e, a^2-f\in U(R), ef=fe$ and $a^2e, a^2f\in R^{qnil}.$ Thus,
$$\begin{array}{lll}
1-(1-e)f&=&1-(1-e)(a^2-e)^{-1}(a^2-e)f\\
&=&1-(1-e)(a^2-e)^{-1}a^2f\\
&=&1-ca^2f,
\end{array}$$ where $c=(1-e)(a^2-e)^{-1}$. As $p\in comm^2(a)$, we
have $c\in comm(a^2f)$. It follows from $a^2f\in R^{qnil}$ that
$1-ca^2f\in U(R)$. Therefore
$$1-(1-e)f=1-(1-e)^2f^2=\big(1-(1-e)f\big)\big(1+(1-e)f\big),$$ and
so $1+(1-e)f=1$. This implies that $f=ef$. Likewise, $e=fe$.
Accordingly, $e=ef=fe=f$, and then $p=1-e=1-f=q$, hence the result.\end{proof}

If $R$ is not only a ring, but a Banach algebra, the converse of precious proposition can be guaranteed.

\begin{cor} Let $A$ be a Banach algebra and $a\in A$. Then the following are equivalent:\end{cor}
\begin{enumerate}
\item [(1)]{\it $a$ has generalized Hirano inverse.}
\vspace{-.5mm}
\item [(2)]{\it There exists an idempotent $p\in A$ such
that
$$pa=ap, a^2-p\in R^{qnil}.$$}\vspace{-.5mm}
\end{enumerate}
\begin{proof} $(1)\Rightarrow (2)$ This is obvious by Proposition 2.7.

$(2)\Rightarrow (1)$ As $a^2-p\in A^{qnil}$, $a^2+1-p\in U(A)$. Set $e=1-p$ and $b=(a^2+1-p)^{-1}(1-e)$.
As in the proof of Theorem 2.5, there exists $b\in A$ such
that $$b=bab, ba=ab, a^2-ab\in A^{qnil}.$$ Similarly to Theorem 2.1, we see that $a-a^2b\in A^{qnil}$.
In light of~\cite[Theorem 7.5.3]{H}, $a\in R$ has generalized Drazin inverse, and so $b\in comm^2(a)$. This completes the proof.
\end{proof}

\section{Matrices over projective-free rings}

For any commutative ring $R$, we note that
$$M_2(R)^{qnil}=\{ A\in M_2(R)~|~A^2\in M_2\big(J(R)\big)\} (\mbox{see~\cite[Example 4.3]{W}}).$$
The aim of this section is to completely determine when a $2\times 2$ matrices over projective-free rings has generalized Hirano inverse.

\begin{thm} Let $R$ be a projective-free ring, and let $A\in M_2(R)$. If $A$ has generalized Hirano inverse, then\end{thm}
\begin{enumerate}
\item [(1)]{\it $A^2\in M_2(J(R))$; or}
\vspace{-.5mm}
\item [(2)]{\it $(I_2-A^2)^2\in M_2(J(R))$; or}\vspace{-.5mm}
\item [(3)]{\it $\chi (A^2)$ has a root in $J(R)$ and a root in $1+J(R)$.}\vspace{-.5mm}
\end{enumerate}
\begin{proof} By virtue of Theorem 2.5, we may write $A^2=E+W$, where $E^2=E, W\in M_2(R)^{qnil}$ and $E\in comm^2(A)$. In light of~\cite[Proposition 11.4.9]{CH},
there exists $U\in GL_2(R)$ such that $UEU^{-1}=0, I_2$ or $\left(
\begin{array}{cc}
1&0\\
0&0
\end{array}
\right)$. Then $A^2\in M_2(R)^{qnil}, I_2-A^2\in M_2(R)^{qnil}$, or $\left(
\begin{array}{cc}
1&0\\
0&0
\end{array}
\right)$.

Case 1. $A^2\in M_2(R)^{qnil}$. Then $A\in M_2(R)^{qnil}$, and so $A^2\in M_2(J(R))$.

Case 2. $I_2-A^2\in M_2(R)^{qnil}$. Then $(I_2-A^2)^2\in M_2(J(R))$.

Case 3. There exist $U\in GL_2(R)$ such that $UEU^{-1}=\left(
\begin{array}{cc}
1&0\\
0&0
\end{array}
\right)$. Hence, $UA^2U^{-1}=UEU^{-1}+UWU^{-1}$. It follows from
$(UEU^{-1})$ $(UWU^{-1})=(UWU^{-1})(UEU^{-1})$ that $UWU^{-1}=\left(
\begin{array}{cc}
\lambda&0\\
0&\mu
\end{array}
\right)$ for some $\lambda,\mu\in J(R)$. Hence, $UA^2U^{-1}=
\left(
\begin{array}{cc}
1+\lambda&0\\
0&\mu
\end{array}
\right)$. Therefore $\chi (A^2)$ has a root $\mu$ in $J(R)$ and a root $1+\lambda$ in $1+J(R)$, as desired.\end{proof}

\begin{cor} Let $A\in M_2({\Bbb Z})$. Then $A$ has generalized Hirano inverse if and only if\end{cor} \vspace{-.5mm}
\begin{enumerate}
\item [(1)]{\it $A^2=0$, or}
\vspace{-.5mm}
\item [(2)]{\it $(I_2-A^2)^2=0$, or}
\vspace{-.5mm}
\item [(3)]{\it $A^2=A^4$.}\vspace{-.5mm}
\end{enumerate}\begin{proof}  $\Longleftarrow$ This is clear by Theorem 2.5.

$\Longrightarrow$ Since $J({\Bbb Z})=0$, by Theorem 3.1, we have

Case 1. $A^2=0$.

Case 2. $(I_2-A^2)^2=0$.

Case 3. $\chi (A^2)$ has a root $0$ and a root $1$. By Cayley-Hamilton
Theorem, $A^2(A^2-I_2)=0$; hence, $A^2=A^4$. In light of Theorem 2.5,
$A\in M_2({\Bbb Z})$ has generalized Hirano inverse.

Therefore we complete the proof.\end{proof}

\begin{exam} Let ${\Bbb
Z}_{(2)}=\{ \frac{m}{n}~|~ m,n\in {\Bbb Z}, (m,n)=1, 2\nmid n\}$, and let
$A=\left(
\begin{array}{cc}
1&2\\
3&4\end{array} \right)$. Then $A\in M_2({\Bbb Z}_{(2)})$ has no generalized Hirano inverse.\end{exam}
\begin{proof} Clearly, ${\Bbb
Z}_{(2)}$ is a commutative local ring with $J({\Bbb
Z}_{(2)})=2{\Bbb Z}_{(2)}$. Since $A^2=\left(
\begin{array}{cc}
7&10\\
15&22\end{array} \right)$, we see that $A^2, (I_2-A^2)^2\not\in M_2(J(R))$.
In addition,
$tr(A^2)=29$ and $det(A^2)=4$. Taking $\sigma: {\Bbb Z}_{(2)}\to
\mathbb{Q}$ to be the natural map, and $p(x)=x^2-29x+4\in {\Bbb
Z}_{(2)}[x]$. Clearly, $p^*(x)=x^2-29x+4\in \mathbb{Q}[x]$ is
irreducible, and so $x^2-29x+4=0$ is not solvable in ${\Bbb
Z}_{(2)}$. Hence, $x^2-tr(A^2)x+det(A^2)=0$ is not solvable in ${\Bbb Z}_{(2)}$.
In light of Theorem 3.1, $A\in M_2({\Bbb Z}_{(2)})$
has no generalized Hirano inverse.
\end{proof}

\begin{thm} Let $R$ be
a projective-free ring. Then $A\in M_2(R)$ has generalized Hirano inverse if and only if\end{thm} \vspace{-.5mm}
\begin{enumerate}
\item [(1)]{\it $det(A), tr(A^2)\in J(R)$; or}
\vspace{-.5mm}
\item [(2)]{\it $det^2(A)\in 1+J(R), tr(A^2)\in 2+J(R)$; or }
\vspace{-.5mm}
\item [(3)]{\it $det(A)\in J(R), tr(A^2)\in 1+J(R)$, and $x^2-tr(A^2)x+det^2(A)=0$ is solvable.}
\end{enumerate}
\begin{proof} $\Longleftarrow$ Case 1. $det(A), tr(A^2)\in J(R)$. By Cayley-Hamilton Theorem, we have $A^2-tr(A^2)A^2+det(A^2)I_2=0$; hence, $A^2=tr(A^2)A^2-det(A^2)I_2\in J(M_2(R))$. Thus, $A^4\in M_2(J(R))$, and so
$A^2\in M_2(R)^{qnil}$. Thus, $A$ has generalized Hirano inverse, by Theorem 2.5.

Case 2. $det^2(A)\in 1+J(R), tr(A^2)\in 2+J(R)$. Then $tr(I_2-A^2)=2-tr(A^2)\in J(A)$ and $det(I_2-A^2)=1-tr(A^2)+det(A^2)\in J(R)$. As in Case 1,
$I_2-A^2\in M_2(R)^{qnil}$. In light of Theorem 2.5, $A$ has generalized Hirano inverse.

Case 3. $det(A)\in J(R), tr(A^2)\in 1+J(R)$ and $x^2-tr(A^2)x+det(A^2)=0$ is solvable.
Write $A^2=\left(
\begin{array}{cc}
a&b\\
c&d
\end{array}
\right)$. Then $a+d=tr(A^2)\in U(R)$. Hence, $a$ or $d$ are invertible. Say $a\in U(R)$.
Suppose $x^2-tr(A^2)x+det(A^2)=0$ is solvable. Say $x_1^2-tr(A^2)x_1+det(A^2)=0$. Set $x_2=tr(A^2)-x_1$. Then $x_2^2-tr(A^2)x_2+det(A^2)=0$. As $x_1x_2=det(A^2)\in J(R)$, we see that $x_1\in J(R), x_2\in U(R)$ or $x_1\in U(R), x_2\in J(R)$. Say $x_1\in J(R),x_2\in U(R)$. Then $x_2^2=tr^2(A^2)+(x_1-2tr(A^2))x_1\in 1+J(R)$.
Let $U=\left(
\begin{array}{cc}
b&a-x_1\\
x_1-a&c
\end{array}
\right)$. It follows from $det(U)=atr(A^2)+(x_1^2-2ax_1-det(A^2))\in U(R)$ that $U\in GL_2(R)$. One easily checks that $$U^{-1}A^2U=\left(
\begin{array}{cc}
x_1&\\
&x_2
\end{array}
\right),$$ and so
$$U^{-1}A^2U=\left(
\begin{array}{cc}
x_1&\\
&x_2
\end{array}
\right)
=\left(
\begin{array}{cc}
0&\\
&1
\end{array}
\right)+\left(
\begin{array}{cc}
x_1&\\
&x_2-1
\end{array}
\right).$$ If $C=(c_{ij})\in comm\left(
\begin{array}{cc}
x_1&\\
&x_2
\end{array}
\right)$, then $c_{12}(x_1-x_2)=c_{21}(x_1-x_2)=0$. As
 $x_1-x_2\in U(R)$, we get $c_{12}=c_{21}=0$. Therefore $C\left(
\begin{array}{cc}
0&\\
&1
\end{array}
\right)=\left(
\begin{array}{cc}
0&\\
&1
\end{array}
\right)C$. This shows that $\left(
\begin{array}{cc}
0&\\
&1
\end{array}
\right)\in comm^2\left(
\begin{array}{cc}
x_1&\\
&x_2
\end{array}
\right)$. In light of Theorem 2.5, $A\in M_2(R)$ has generalized Hirano inverse.

$\Longrightarrow$ In view of Theorem 3.1,
$A^2\in M_2(J(R))$; or $(I_2-A^2)^2\in M_2(J(R))$; or $\chi (A^2)$ has a root in $J(R)$ and a root in $1+J(R)$.

Case 1. $A^2\in M_2(J(R))$. Hence, $det(A), tr(A^2)\in J(R)$.

Case 2. $(I_2-A^2)^2\in M_2(J(R))$. Then $\overline{I_2-A^2}\in N(M_2(R/J(R))$.
In light of~\cite[???]{???}, $\chi (\overline{I_2-A^2})\equiv t^2 \big(mod~ N(R/J(R)\big)$. It follows from $N(R/J(R))$ $=$ $\overline{0}$ that
$\chi(I_2-A^2)-t^2\in J(R)[t]$. As $\chi (I_2-A^2)=t^2-tr(I_2-A^2)t+det(I_2-A^2)$, we have
$tr(I_2-A^2), det(I_2-A^2)\in J(R)$. Hence, $tr(A^2)\in 2+J(R)$. Since $det(I_2-A^2)=1-tr(A^2)+det^2(A)$, we get
$det^2(A)\in 1+J(R),$ as desired.

Case 3. $\chi (A^2)$ has a root $\alpha$ in $J(R)$ and a root $\beta$ in $1+J(R)$. Since $\chi (A^2)=x^2-tr(A^2)x+det^2(A)$, we see that $det^2(A)=\alpha\beta\in J(R)$, and so $det(A)\in J(R)$. Moreover, $tr(A^2)=\alpha+\beta\in 1+J(R)$. Therefore we complete the proof.\end{proof}

\begin{cor}  Let $R$ be a projective-free ring, and let $A\in M_2(R)$. If
$\chi(A^2)=(t-\alpha)(t-\beta)$, where
$\alpha,\beta\in J(R)\bigcup (1+J(R))$, then $A$ has generalized Hirano inverse.\end{cor}
\begin{proof} Since $\chi(A^2)=(t-\alpha)(t-\beta)$, $det(A^2)=\alpha\beta$ and $tr(A^2)=\alpha+\beta$.

Case 1. $\alpha,\beta\in J(R).$ Then $det(A), tr(A^2)\in J(R)$.

Case 2. $\alpha,\beta\in 1+J(R).$ Then $det^2(A)\in 1+J(R), tr(A^2)\in 2+J(R)$.

Case 3. $\alpha\in J(R), \beta\in 1+J(R)$. Then $det(A)\in J(R), tr(A^2)\in 1+J(R)$.

Case 4. $\alpha\in 1+J(R),\beta\in J(R).$  Then $det(A)\in J(R), tr(A^2)\in 1+J(R)$.

Therefore $A$ has generalized Hirano inverse, by Theorem 3.4.\end{proof}

\begin{cor} Let $R$ be
a projective-free ring. If $\frac{1}{2}\in R$, then $A\in M_2(R)$ has Hrano inverse if and only if\end{cor} \vspace{-.5mm}
\begin{enumerate}
\item [(1)]{\it $det(A), tr(A^2)\in J(R)$; or}
\vspace{-.5mm}
\item [(2)]{\it $det^2(A)\in 1+J(R), tr(A^2)\in 2+J(R)$; or }
\vspace{-.5mm}
\item [(3)]{\it $det(A)\in J(R), tr(A^2)\in 1+J(R)$ and there exists $u\in U(R)$ such that $tr^2(A^2)-4det^2(A)=u^2$.}
\end{enumerate}
\begin{proof} Suppose that $det(A)\in J(R), tr(A^2)\in 1+J(R)$. If $x^2-tr(A^2)x+det^2(A)=0$ is solvable, then this equation has a root $x_1$. Set $x_2=tr(A^2)-x_1$. Then $det(A^2)=x_1(tr(A^2)-x_1)=x_1x_2\in J(R)$. Set $u=x_1-x_2$. Then $u\in U(R)$. Therefore $tr^2(A^2)-4det^2(A)=(x_1+x_2)^2-4x_1x_2=(x_1-x_2)^2=u^2$. Conversely, assume that $tr^2(A^2)-4det^2(A)=u^2$ for a $u\in U(R)$. Then $\frac{tr(A^2)+u}{2}$ is a root of $x^2-tr(A^2)x+det^2(A)=0$. In light of Theorem 3.4, we complete the proof.\end{proof}

\begin{exam} Let ${\Bbb
Z}_{(2)}=\{ \frac{m}{n}~|~ m,n\in {\Bbb Z}, (m,n)=1, 2\nmid n\}$, and let
$A=\left(
\begin{array}{cc}
5&6\\
3&2\end{array} \right)\in M_2({\Bbb Z}_{(2)})$. Then $A\in M_2({\Bbb Z}_{(2)})$ has generalized Hirano inverse.\end{exam}
\begin{proof} Clearly, ${\Bbb
Z}_{(2)}$ is a commutative local ring with $J({\Bbb
Z}_{(2)})=2{\Bbb Z}_{(2)}$. Clearly, $A^2=\left(
\begin{array}{cc}
43&42\\
21&22\end{array} \right)$. Then $det(A^2)=64\in J({\Bbb Z}_{(2)}),$ $tr(A^2)=65\in 1+J({\Bbb Z}_{(2)})$. Thus, $det(A)\in J({\Bbb Z}_{(2)})$, and that $x^2-tr(A^2)x+det^2(A)=0$ has a root $1$. Therefore we are done by Theorem 3.4.
\end{proof}

\section{Cline's formula}

In ~\cite[Theorem 2.1]{LCC2}, Liao et al. proved Cline's formula for generalized Drazin inverse. Lian and Zeng extended this result and proved that if $aba=aca$ then $ac$ has generalized Drazin inverse if and only if $ba$ has generalized Drazin inverse (see~\cite[Theorem 2.3]{LZ}). The aim of this section is to generalize Cline¡¯s formula from generalized Drazin inverses to generalized Hirano inverses.

\begin{thm} Let $R$ be a ring, and let $a,b,c\in R$. If $aba=aca$, Then $ac$ has generalized Hirano inverse if and only if $ba$ has generalized Hirano inverse. In this case, $(ba)^{h}=b((ac)^{h})^2a$.
\end{thm}
\begin{proof} Suppose that $ac$ has generalized Hirano inverse and $(ac)^{h}=d$. Let $e=bd^2a$ and $f\in comm(ba)$. Then $$fe=fb(dacd)^2a=fbacacd^4a=babafcd^4a=b(acafc)d^4a.$$
Also we have
 $$\begin{array}{lll}
 ac(acafc)&=&ababafc=afbabac=afbacac\\
 &=&abafcac=(acafc)ac.
 \end{array}$$
 As $d\in comm^2(ac)$, we see that $(acafc)d=d(acafc)$. Thus, we conclude that
 $$\begin{array}{lll}
 fe&=&b(acafc)d^4a=bd^4(acafc)a\\
 &=&bd^4abafca=bd^4afbaca\\
 &=&bd^4afbaba=bd^4ababaf\\
 &=&bd^4acacaf=bd^2af=ef.
 \end{array}$$
 This implies that $e\in comm^2(ba)$.
 We have
 $$\begin{array}{lll}
 e(ba)e&=&bd^2a(ba)bd^2a=bd^2ababacd^3a\\
 &=&bd^2(ac)^3d^3a=bd^2a=e.
 \end{array}$$
 Let $p=ac-acd^2$ then,
  $$pac=acac-acd^2ac=acac-dac=(ac)^2-acd$$
  that is contained in $R^{qnil}$. Also we have
  $$\begin{array}{lll}
  (ba)^2-bae&=&baba-babd^2a=baba-babacd^2da\\
  &=&baca-bacacd^2da=b(ac-acd^2)a=bpa.
  \end{array}$$
  Since
  $$abpa=ab(ac-acd^2)a=ac(ac-acd^2)a=acpa,$$
  we deduce that
  $$(pa)b(pa)=(pa)c(pa).$$
  Then by \cite[Lemma 2.2] {LZ}, $bpa\in R^{qnil}$. Hence $e$ is generalized Hirano inverse of $ba$ and as the generalized Hirano inverse is unique we deduce that $e=bd^2a=(ba)^{h}.$
  \end{proof}

\begin{cor} Let $R$ be a ring and $a,b\in R$. If $ab$ has generalized Hirano inverse, then so has $ba$, and $(ba)^{h} = b((ab)^{h})^2a.$\end{cor}
\begin{proof} This follows directly from Theorem 4.1 as $aba=aba.$\end{proof}

\begin{cor} Let $R$ be a ring, and let $a,b\in R$. If $(ab)^k$ has generalized Hirano inverse, then so is $(ba)^k$.
\end{cor}
\begin{proof} Let $k=1$, then by Corollary 4.2, $ab$ has generalized Hirano inverse if and only if $ba$ has generalized Hirano inverse. Now let $k\geq 2$, as $(ab)^k=a(b(ab)^{k-1})$, then by Corollary 4.2, $(b(ab)^{k-1})a$ has generalized Hirano inverse, so $(ba)^k$ has generalized Hirano inverse.\end{proof}

\begin{lem} Let $A$ be a Banach algebra, $a,b\in A$ and $ab=ba$.\end{lem}
\begin{enumerate}
\item [(1)]{\it If $a,b\in R^{qnil}$, then $a+b\in R^{qnil}$.}
\vspace{-.5mm}
\item [(2)]{\it If $a$ or $b\in R^{qnil}$, then $ab\in R^{qnil}$.}
\end{enumerate}
\begin{proof} Seebb~\cite[Lemma 2.10 and Lemma 2.11]{ZMC}.\end{proof}

\begin{thm} Let $A$ be a Banach algebra , and let $a, b\in A$. If $a,b$ have generalized Hirano inverses and $ab=ba$. Then $ab$ has generalized Hirano inverse and $$(ab)^{h}=a^{h}b^{h}.$$
\end{thm}
\begin{proof} In order to prove that $(ab)^{h}=a^{h}b^{h}$, we need to prove
 $$a^{h}b^{h}aba^{h}b^{h}=a^{h}b^{h},$$
 $a^{h}b^{h}\in comm^2(ab)$ and $(ab)^2-aba^{h}b^{h}\in R^{qnil}.$

As $a^{h}aa^{h}=a^{h}$ and $ b^{h}bb^{h}=b^{h}$, also $a$ can be commuted with $a^{h}, b, b^{h}$ and $b$ can be commuted with $a, a^{h}, b^{h}$,
we have $$ a^{h}b^{h}aba^{h}b^{h}=a^{h}b^{h}.$$ Let $e\in comm^2(ab)$, then $a^{h}b^{h}e=a^{h}eb^{h}=ea^{h}b^{h}$, as $a^{h} , b^{h}\in com(ab).$
Note that
$$\begin{array}{ll}
&(ab)^2-aba^{h}b^{h}\\
=&a^2(b^2-bb^{h})+b^2(a^2-aa^{h})-(a^2-aa^{h})(b^2-bb^{h})
\end{array}$$
that is in $R^{qnil}$ by Lemma 4.4, which implies that $(ab)^2-aba^{h}b^{h}\in R^{qnil}.$
\end{proof}

\begin{cor} Let $A$ be a Banach algebra. If $a\in A$ has generalized Hirano inverse, then $a^n$ has generalized Hirano inverse and $(a^n)^{h}=(a^{h})^n$ for all $n\in {\Bbb N}$.\end{cor}
\begin{proof} It easily follows from Theorem 4.5 and using induction on $n$.\end{proof}

We note that the converse of the previous corollary is not true.

\begin{exam} Let ${\Bbb Z}_5={\Bbb Z}/5{\Bbb Z}$ be the ring of integers modulo $5$. Then $-2\in {\Bbb Z}_5$ has no generalized Hirano inverse. If $b$ is the generalized Hirano inverse of $-2$, then $-2b^2=b$ which implies that $b=0$ or $b=2$, in both cases $(-2)^2-(-2)b$ is not quasinilpotent, a contradiction.
But $(-2)^2=-1\in {\Bbb Z}$ has generalized Hirano inverse $1$.\end{exam}

\section{Additive properties}

In this section, we are concern on additive properties of generalized Hirano inverses in a Banach algebra. Let $p\in R$ be an idempotent, and let $x\in R$. Then we write $$x=pxp+px(1-p)+(1-p)xp+(1-p)x(1-p),$$ and induce a representation given by the matrix
$$x=\left(\begin{array}{cc}
pxp&px(1-p)\\
(1-p)xp&(1-p)x(1-p)
\end{array}
\right)_p,$$ and so we may regard such matrix as an element in $R$. The following lemma is crucial.

\begin{lem} Let $R$ be a Banach algebra, let $a\in A$ and let $$x=\left(\begin{array}{cc}
a&c\\
0&b
\end{array}
\right)_p,$$ relative to $p^2=p\in A$. If $a\in pAp$ and $b\in (1-p)A(1-p)$ have generalized Hirao inverses, then so is $x$ in $A$.
\end{lem}
\begin{proof} In view of Corollary 2.3, $a$ and $b$ have generalized Drazin inverses, and that $a^2-aa^d, b^2-bb^d\in A^{qnil}$. Hence,
In view of~\cite[Lemma 2.1]{ZMC}, we can find $u\in pA(1-p)$ such that $$e=\left(\begin{array}{cc}
a^d&u\\
0&b^d
\end{array}
\right)_p\in comm^2(x), e=e^2x~\mbox{and}~x-x^2e\in A^{qnil}.$$ Thus,
$$x^2-xe=\left(\begin{array}{cc}
a^2-aa^d&v\\
0&b^2-bb^d
\end{array}
\right)_p$$ for some $v\in A$. Clearly, $v(a^2-aa^d)=0$ and $v^2=0$. It follows by ~\cite[Lemma 2.1]{C} that $(a^2-aa^d)+v\in A^{qnil}$. Furthermore,
$(b^2-bb^d)\big((a^2-aa^d)+v\big)=0$. By using ~\cite[Lemma 2.1]{C} again, $(a^2-aa^d)+v+(b^2-bb^d)\in A^{qnil}$. Therefore $x^2-xe\in A^{qnil}$.
Accordingly, $x$ has generalized Hirao inverses, as asserted.\end{proof}

\begin{lem} Let $A$ be a Banach algebra. Suppose that $a\in A^{qnil}$ and $b\in A$ has generalized Hirano inverses. If $$a=ab^{\pi},b^{\pi}ba=b^{\pi}ab,$$ then $a+b$ has generalized Hirano inverse.\end{lem}
\begin{proof} Case 1. $b\in A^{qnil}$. Then $b^{\pi}=1$, and so it follows from $b^{\pi}ba=b^{\pi}ab$ that $ab=ba$.
In light of Lemma 4.4, $a+b\in A^{qnil}$. By using Lemma 4.4 again, $(a+b)^2\in A^{qnil}$. Therefore $a+b$ has generalized Hirano inverse by Theorem 2.5.

Case 2. $b\not\in A^{qnil}$. In view of Theorem 2.1, $b$ has generalized Drazin inverse, and so we have
$$b=\left(\begin{array}{cc}
b_1&0\\
0&b_2
\end{array}
\right)_p, a=\left(\begin{array}{cc}
a_{11}&a_{1}\\
a_{21}&a_2
\end{array}
\right)_p,$$ where $b_1\in U(pAp), b_2\in ((1-p)A(1-p))^{qnil}\subseteq A^{qnil}$. From $a=ab^{\pi}$, we see that $a_{11}=a_{21}=0$, and so
$$a+b=\left(\begin{array}{cc}
b_1&a_1\\
0&a_2+b_2
\end{array}
\right)_p.$$ It follows from $b^{\pi}ba=b^{\pi}ab$ that $a_2b_2=b_2a_2$. By Lemma 4.4, $a_2+b_2\in A^{qnil}$, and then $(a_2+b_2)^2\in A^{qnil}$. In light of Theorem 2.5, $a_2+b_2$ has generalized Hirano inverse. Since $b$ has generalized Hirano inverse in $A$, one easily checks that $b_1=pbp$ has generalized inverse. According to Lemma 5.1, $a+b$ has generalized Hirano inverse, thus completing the proof.\end{proof}

\begin{thm} Let $A$ be a Banach algebra. If $a,b\in A$ have generalized Hirano inverses and satisfy
$$a=ab^{\pi}, b^{\pi}ba^{\pi}=b^{\pi}b, b^{\pi}a^{\pi}ba=b^{\pi}a^{\pi}ab,$$ then $a+b$ has generalized Hirano inverse and
$$\begin{array}{lll}
(a+b)^{h}&=&\big(b^{h}+\sum\limits_{n=0}^{\infty}(b^{h})^{n+2}a(a+b)^n\big)a^{\pi}\\
&-&\sum\limits_{n=0}^{\infty}\sum\limits_{k=0}^{\infty}(b^{h})^{n+2}a(a+b)^n(a^{h})^{k+2}b(a+b)^{k+1}\\
&+&\sum\limits_{n=0}^{\infty}(b^{h})^{n+2}a(a+b)^na^{h}b-
\sum\limits_{n=0}^{\infty}b^{h}a(a^{h})^{n+2}b(a+b)^n.
\end{array}$$\end{thm}
\begin{proof}  Case 1. $b\in A^{qnil}$. Then $b^{\pi}=1$, and so we obtain the result by applying Lemma 5.2.

Case 2. $b\not\in A^{qnil}$. We have
$$b=\left(\begin{array}{cc}
b_1&0\\
0&b_2
\end{array}
\right)_p, a=\left(\begin{array}{cc}
a_{11}&a_{1}\\
a_{21}&a_2
\end{array}
\right)_p,$$ where $b_1\in U(pAp), b_2\in ((1-p)A(1-p))^{qnil}\subseteq A^{qnil}$. It follows from $a=ab^{\pi}$ that $a_{11}=a_{21}=0$, and so
$$a+b=\left(\begin{array}{cc}
b_1&a_1\\
0&a_2+b_2
\end{array}
\right)_p.$$ Since $b^{\pi}ba^{\pi}=b^{\pi}b, b^{\pi}a^{\pi}ba=b^{\pi}a^{\pi}ab$, we see that $a_2^{\pi}a_2b_2=a_2^{\pi}b_2a_2$. Now, it follows that
$a_2+b_2\in (1-p)A(1-p)$ has generalized Hirano inverse. As $b_1=pbp$ has generalized Hirano inverse, it follows by Lemma 5.1 that $a+b$ has generalized Hirano inverse. According to Corollary 2.2, we see that $(a+b)^{h}=(a+b)^{d}, a^{h}=a^{d}$ and $b^{h}=b^{d}$. Therefore we complete the proof, by ~\cite[Theorem 2.2]{C}.\end{proof}

\begin{cor} Let $A$ be a Banach algebra and $a,b\in A$. If $a,b\in A$ have generalized Hirano inverses and satisfy
$$ab=ba, a=ab^{\pi}, b^{\pi}=ba^{\pi}=b^{\pi}b,$$ then $a+b$ has generalized Hirano inverse and
$(a+b)^{h}=b^{h}.$ \end{cor}
\begin{proof} Since $a=ab^{\pi}$, we see that $ab^h=0$. Thus, $(b^{h})^{n+2}a(a+b)^n=a(b^{h})^{n+2}(a+b)^n=0.$
Therefore the result follows by Theorem 5.3.\end{proof}

\begin{cor} Let $A$ be a Banach algebra and $a,b\in A$. If $a,b$ have generalized Hirano inverses and $ab=ba=0$, then $a+b$ has generalized Hirano inverse and $(a+b)^{h}=a^{h}+b^{h}.$\end{cor}
\begin{proof} This is obvious by Theorem 5.3.\end{proof}

\end{document}